\documentclass[12pt]{article}
\usepackage{amsfonts}
\usepackage[usenames]{color}
\usepackage{graphicx}
\usepackage{amsmath}

\newtheorem{theorem}{Theorem}

\newtheorem{corollary}[theorem]{Corollary}

\newtheorem{lemma}[theorem]{Lemma}

\newtheorem{proposition}[theorem]{Proposition}
\newtheorem{remark}[theorem]{Remark}

\newenvironment{proof}[1][Proof]{\textbf{#1.} }{\ \rule{0.5em}{0.5em}}

\begin{document}

\title{{\LARGE A density for the local time of  the Brox diffusion
\footnote{This research was partially supported by the CONACYT and
the Universidad de Costa Rica.}}}
\author{
Jonathan Gutierrez-Pav\'{o}n
\thanks{Department of Mathematics, University of Costa Rica, San Jos\'e, Costa Rica.  Email: jonathan.gutierrez@ucr.ac.cr} and
Carlos G. Pacheco
\thanks{Department of Mathematics, CINVESTAV, Mexico City, Mexico Email: cpacheco@math.cinvestav.mx}}

\maketitle

\begin{abstract}
We give explictly the probability density of the local time of the Brox diffusion at first passage times.
Such formula is used to find the moments and to related the minima and maxima of the environment to the most and least visted points of the diffusion.
\end{abstract}

{\bf 2000 Mathematics Subject Classification: }
\\



\section{Introduction}

When studying processes with a random environment, it is natural to analyze de effect of the environment on the behaviour of the processs.
In fact, one is interested on determining, up to some confidence, the behaviour of the trayectories by knowing something of the environment.
A very well known situtation of this type arises with models such as the so-called Sinai's random walk and the Brox diffusion, 
where information of the environment can say where is mainly localized the particle, see for instance \cite{Revesz, Brox}.

In this paper we focus on this kind of question for the Brox diffusion. 
It is known that the environment gives rise to a stochastic process, sometimes called the potential, 
and it is already known as well how the minimum values of such potential determines the places where the particles acumulates more local time, see \cite{Cheliotis}.
The places in the state space where the local time takes high values are said to be favorite points.

For these reasons we are interested on studying the local time and its probability densities.
One has to remember that for the Brownian motion the local times can be related to specific stochastic processes such as the Bessel diffusion; this is the content of the so-called Ray-Knight theorems.
What we do here is to calculate explicitly the probability density of the local time for the Brox diffusion under certaing stopping times.
It is interesting to see how the density is a mixture of a absolutely continuous density and a discrete one.
A very straight forward application is finding moments of the local times where one can easly read the connection between minima of the potential and favorite points of the diffusion, see Remark \ref{NotaPF}.

\section{Preliminaries}

Intuitively speaking the Brox diffusion was originally described with the following stochastic differential equation
\begin{equation} \label{ecuacion informal de brox}
dX_{t}=-\frac{1}{2}  \,W'(X_{t})\,dt+dB_{t},
\end{equation}
where $B:= \{B_{t}: t\geq 0  \}$ represents the standard Brownian motion, and $W:= \{ W(x) : x \in \mathbb{R} \}$ is a two-sided Brownian motion. 
The processes $B$ and $W$ are independent from each other. 
Here $W'$ denotes the derivative of $W$, and usually $W'$ is known as the white noise.

To give a rigourous meaning of the expression (\ref{ecuacion informal de brox}) we can use the associated generator of the Markov process. 
More precisely one can say that the process $X:= \{ X_{t}:t\geq 0 \}$ is associated with the infinitesimal generator
\begin{equation*} \label{brox operator informal}
Lf(x):=\displaystyle\frac{1}{2}\;  e^{W(x)}\,\frac{d}{dx} \left(
e^{-W(x)}\,\frac{df(x)}{dx}\right).
\end{equation*}
In this case, the scale function is
\begin{equation}\label{brox funcion escala}
s(x):=\displaystyle\int_{0}^{x}\,e^{W(y)}\,dy,
\end{equation}
and the speed measure is
\begin{equation}\label{broz medida velocidad}
m(A):=\displaystyle\int_{A} 2e^{-W(y)}\,dy, \; \; \mbox{for Borel
sets}\; A \subseteq \mathbb{R}.
\end{equation}
Therefore we can represent the operator $L$ in the following way
$$Lf=\displaystyle \frac{d}{dm}\frac{d}{ds}f.$$

When leaving fixed a trajectory of $W$, the process $X$ associated with this operator is indeed a diffusion, thus, one can give an explicity formula of the process $X$ in the following way:
\begin{equation} \label{Xt}
X_{t}=s^{-1}(B_{T^{-1}_{t}}),
\end{equation}
where
$$T_{t}:=\displaystyle\int_{0}^{t}e^{-2W(s^{-1}(B_{u}))}\,du.$$
Now, let us say some facts on local times.\\

Let $L_{X}(t,x)$ be the local time of the Brox diffusion at time $t$ at the site $x$. 
It is known that $L_{X}:=\{ L_{X}(t,x): t\geq 0, x\in\mathbb{R} \}$ is jointly continuous stochastic process such that for any measurable function $f$ and for any $t\geq 0$:
\begin{equation} \label{ecuacion de local time brox}
\displaystyle\int_{0}^{t}f(X_{s})ds = \int_{-\infty}^{\infty}
f(x)L_{X}(t,x)dx.
\end{equation}
By substituting (\ref{Xt}) in (\ref{ecuacion de local time brox}) it is easy see that
\begin{equation}\label{definicion local time brox}
L_{X}(t,x)=e^{-W(x)}L_{B}(T^{-1}(t),s(x)),
\end{equation}
where $L_{B}$ represent the local time of the Brownian motion, see \cite{Shi}.
We use previous formula to give a explicit expression of local time of the Brox process.

Finally, we briefly mention the statement of the Ray-Knight theorem, see for instance \cite{RevuzYor}.
It says that $\{L_{B}(\sigma(1),1-a), a\in [0,1] \}$ is the same in distribution as a two dimensional squared Bessel process $\{Z_{a}, \ 0\leq a \leq 1\}$, where $\sigma(1)=\inf \{t: B_{t}=1 \}$.
In general, it is also true that for fixed $a < b$, the local time $\{L_{B}(\sigma(b),a) \}$ has the same distribution as the random variable $Z_{b-a}$.
In particular, for fixed $a < b$, the distribution of $L_{B}(\sigma(b),a)$ is an exponential random variable with mean 2.

\section{Local time of Brox diffusion}

Let us now define the random variable that we are going to analyze.
Take the stopping time $\tau(b):= \inf \{t: X_{t}=b \}$.
Our goal is to study the stochastic process $\{L_{X}(\tau(b),a), b>a\}$, for $a,b$ fixed.

\begin{proposition}
For $0 < a <b $ fixed, the random variable $L_{X}(\tau(b),a)$ is exponential with parameter $\lambda:= \displaystyle\frac{e^{W(a)}}{2(s(b)-s(a))}.$
\end{proposition}
\begin{proof}
Using (\ref{definicion local time brox}) we have that $L_{X}(\tau(b),a)= e^{-W(a)}L_{B}(T^{-1}(\tau(b)),s(a))$. 
One can see that $T^{-1}(\tau(b))=\sigma(s(b))$.

Then, by the Ray-Knight theorem, if $Z$ represents a two dimensional squared Bessel process, then we have
\begin{eqnarray*}
P(L_{X}(\tau(b),a) \leq x)&=& P(e^{-W(a)}L_{B}(T^{-1}(\tau(b)),s(a)) \leq x)\\
&=& P(e^{-W(a)}L_{B}(\sigma(s(b)),s(a)) \leq x)\\
&=& P(e^{-W(a)}Z_{s(b)-s(a)}) \leq x)\\
&=& P(e^{-W(a)} (s(b)-s(a))Z_{1} \leq x)\\
&=& P \left(Z_{1} \leq \displaystyle \frac{e^{W(a)}x}{s(b)-s(a)}  \right).
\end{eqnarray*}
Using the fact that $Z_{1}$ is an exponential random variable with mean $2$, we obtain the result.
\end{proof}\\


\begin{lemma}
For any fixed environment $W$, and $a\in\mathbb{R}$ fixed, the process $\{ L_{X}(\tau(t),a): t\geq 0\}$ has independent increments.
\end{lemma}
\begin{proof}
Take times $t_{1}<t_{2}<t_{3}<t_{4}$, and define de increments $I_{1}$ and $I_{2}$ as, 
$$I_{1}:=L_{X}(\tau(t_{2}),a)-L_{X}(\tau(t_{1}),a),$$
and similarly for $I_{2}$ using $t_{3}$ and $t_{4}$ in liu of $t_{1}$ and $t_{2}$, respectively.

Using (\ref{definicion local time brox}) we have that
$$I_{1}=e^{-W(a)}\left( L_{B}(T^{-1}(\tau(t_{2})),a)-L_{B}(T^{-1}(\tau(t_{1})),a) \right)$$
and the same for $I_{1}$.

If $T_{i}:=T^{-1}(\tau(t_{i}))$ for $i=1,2,3,4$, notice that $T_{1}\leq T_{2}\leq T_{3}\leq T_{4}$, and although they depend on $B$, they are stopping times. 
Thus, by the strong Markov property, the random varibles $I_{1}$ and $I_{2}$ are independent because they are measurements on disjoint ''time-windows''. 
\end{proof}\\

We use previous results to find the distribution of the random variable 
$$L_{X}(\tau(c),a)-L_{X}(\tau(b),a),$$ 
where $0 < a < b < c$.

\begin{theorem}
The probability density $f$ of $L_{X}(\tau(c),a)-L_{X}(\tau(b),a)$ is given by 
\begin{equation}\label{Densidad}
f(t)=\lambda (1- \alpha)e^{-\lambda t} + \alpha \delta_{0}(t),
\end{equation}
where 
$$\lambda= \displaystyle\frac{e^{W(a)}}{2(s(c)-s(a))} >0, \text{ and } \alpha=\frac{s(b)-s(a)}{s(c)-s(a)}.$$
Here $\delta_{0}$ represents the Dirac Delta function at the point $0$. 
\end{theorem}

\begin{proof}
Note first that
\begin{equation} \label{ecuacion para usar incrementos indep}
E\left(e^{-tL_{X}(\tau(c),a)}  \right)=
E\left(e^{-t[L_{X}(\tau(c),a)-L_{X}(\tau(b),a)]}  \right)\cdot
E\left(e^{-tL_{X}(\tau(b),a)}  \right).
\end{equation}
Thus,
\begin{equation}
E\left(e^{-t[L_{X}(\tau(c),a)-L_{X}(\tau(b),a)]} \right) =
\displaystyle \frac{E\left(e^{-tL_{X}(\tau(c),a)}  \right)}{
E\left(e^{-tL_{X}(\tau(b),a)}  \right)}.
\end{equation}

On the other hand, using the fact that $L_{X}(\tau(b),a)$ is exponential with parameter 
$\lambda:= \displaystyle\frac{e^{W(a)}}{2(s(b)-s(a))}$, 
then the moment-generating function of $L_{X}(\tau(b),a)$ is the following:
\begin{equation}\label{generadora momentos tiempo tocal}
E\left(e^{-tL_{X}(\tau(b),a)} \right)=
\displaystyle\frac{e^{W(a)}}{e^{W(a)}+2t(s(b)-s(a))}.
\end{equation}
Substituting (\ref{generadora momentos tiempo tocal}) in (\ref{ecuacion para usar incrementos indep}) we obtain the moment-generating function of $L_{X}(\tau(c),a)-L_{X}(\tau(b),a)$
\begin{equation}\label{generadora de momentos del incremento 1}
E\left(e^{-t[L_{X}(\tau(c),a)-L_{X}(\tau(b),a)]}  \right)=
\displaystyle\frac{e^{W(a)}+2t(s(b)-s(a))}{e^{W(a)}+2t(s(c)-s(a))}.
\end{equation}

The previous expression can be written as
\begin{equation}\label{gene. momentos del incremento}
 E\left(e^{-t[L_{X}(\tau(c),a)-L_{X}(\tau(b),a)]}  \right)=
\displaystyle\frac{1}{1+\frac{2(s(c)-s(a))t}{e^{W(a)}}}+
\frac{2(s(b)-s(a))}{e^{W(a)}}\cdot
\frac{t}{1+\frac{2(s(c)-s(a))t}{e^{W(a)}}}.
\end{equation}
We now consider the following formulae:
$$\displaystyle \mathcal{L}^{-1}\left( \frac{1}{1+\frac{2(s(c)-s(a))t}{e^{W(a)}}}  \right)= \frac{e^{W(a)}}{2(s(c)-s(a))}e^{\displaystyle\frac{-e^{W(a)}t}{2(s(c)-s(a))}},$$
$$\displaystyle \mathcal{L}^{-1} \left( \frac{t}{1+\frac{2(s(c)-s(a))t}{e^{W(a)}}}\right)=  \frac{e^{W(a)}}{2(s(c)-s(a))} \delta_{0}(t)- \frac{e^{2W(a)}}{4(s(c)-s(a))^{2}}e^{\displaystyle\frac{-e^{W(a)}t}{2(s(c)-s(a))}}, $$
where $\mathcal{L}^{-1}$ is the inverse Laplace transform, and $\delta_{0}$ represent the Dirac Delta function.

Using the previous formulae, we can obtain the distribution of
$L_{X}(\tau(c),a)-L_{X}(\tau(b),a)$, for $0 < a < b < c$. 

Indeed, if $F$ represents such distribution function, then we have
\begin{equation} \label{distribucion density}
F(t)= \displaystyle\left[1-\frac{s(b)-s(a)}{s(c)-s(a)}\right]\cdot
\left[ 1-e^{\displaystyle\frac{-e^{W(a)}t}{2(s(c)-s(a))}} \right]+
\frac{s(b)-s(a)}{s(c)-s(a)}\cdot H_{0}(t),
\end{equation}
where $H_{0}$ represent the Heaviside function.
Thus we obtain the density (\ref{Densidad}).
\end{proof}

Now, we can use these formulae to calculate explicitly the moments of the random variable $L_{X}(\tau (c),a)-L_{X}(\tau (b),a)$. 
\begin{corollary} 
It holds that 
$$E\left[\left(L_{X}(\tau (c),a)-L_{X}(\tau (b),a)  \right)^{n}\right]=n! \frac{2(s(c)-s(b))}{e^{W(a)}} \left[ \frac{2(s(c)-s(a))}{e^{W(a)}} \right]^{n-1},$$
where $a,b,c$ are three numbers such that $0<a<b<c$. 
\end{corollary}

\begin{proof}
First note that if $f(t)=\frac{1+At}{1+Bt}$, where $A$ and $B$ are constants, then 
\begin{equation}\label{formula de la derivada de f}
\displaystyle\frac{d^{n}f(t)}{dt^{n}}= n! (-1)^{n} (B-A) (1+Bt)^{-n-1} B^{n-1}.
\end{equation}
From the formula (\ref{generadora de momentos del incremento 1}), we can see that the moment-generating function of $L_{X}(\tau(c),a)-L_{X}(\tau(b),a)$ has the form  $f(t)=\frac{1+At}{1+Bt}$ 
with $A= \frac{2(s(b)-s(a))}{e^{W(a)}}$ and $B=\frac{2(s(c)-s(a))}{e^{W(a)}}$. 

Now, using the formulae (\ref{generadora de momentos del incremento 1}) and (\ref{formula de la derivada de f}), 
as well as the fact that the random variable $X$ satisfies 
$$\frac{d^{n}E(e^{-tX})}{dt^{n}}|_{t=0}=(-1)^{n}E(X^{n}),$$  
we have the explicit formula for the moments of $L_{X}(\tau (c),a)-L_{X}(\tau (b),a)$.
\end{proof} \\

Finally, let us mention the following important application. 
In the case where $n=1$ we end up with
\begin{equation} \label{experanza del incremento}
\displaystyle E(L_{X}(\tau (c),a)-L_{X}(\tau (b),a))= \displaystyle
\frac{2(s(c)-s(b))}{e^{W(a)}}.
\end{equation}

\begin{remark}\label{NotaPF}
We can see in the previous expression (\ref{experanza del incremento}) that if $a$ satisfies $W(a)\leq W(t)$ for all $t \in [0,b]$, 
then the expected value of $L_{X}(\tau (c),a)-L_{X}(\tau (b),a)$ is the largest possible.
That is, if $W(a)$ is a minimum of $W$ in the interval $[0,b]$, then it is expected that the random variable $L_{X}(\tau (c),a)-L_{X}(\tau (b),a)$ reachs its largest value.

Similarly, if $W(a)$ is a maximum of $W$ in the interval $[0,b]$,
then it is expected that the random variable $L_{X}(\tau (c),a)-L_{X}(\tau (b),a)$ reachs its smallest value.

In short, the stochastic process $X$ spends more time at minima of the environment, and less time at maxima of the environment.
\end{remark}


\end{document}